\documentclass{amsart}
\usepackage{amsmath,amssymb, amsbsy}
\usepackage[dvips]{graphicx}
\usepackage{color}
\usepackage[latin1]{inputenc}
\usepackage[active]{srcltx}
\usepackage{wrapfig}
\usepackage{graphicx}
\usepackage[usenames,dvipsnames]{pstricks}
\usepackage{epsfig}
 \usepackage{pst-grad} 
 \usepackage{pst-plot} 
\usepackage{epic}

\theoremstyle{plain}

\newtheorem{remark}{Remark}

\numberwithin{equation}{section}
\long\def\salta#1{\relax}
\newcommand{\R}{{I\!\!R}}

\newcommand{\dint}{\dyle\int}

\newcommand{\re}{{I\!\!R}}
\newcommand{\ren}{\re^N}

\newcommand{\dyle}{\displaystyle}

\newcommand{\io}{\int\limits_\O}

\renewcommand{\a }{\alpha }
\renewcommand{\b }{\beta }
\renewcommand{\d }{\delta }
\newcommand{\D }{\Delta }
\newcommand{\e }{\varepsilon }

\newcommand{\g }{\gamma}

\renewcommand{\l }{\lambda }
\renewcommand{\L }{\Lambda }

\newcommand{\s }{\sigma }

\renewcommand{\O }{\Omega }

\newenvironment{pf}{\noindent{\sc Proof}.\enspace}{\rule{2mm}{2mm}\medskip}

\newtheorem{Theorem}{Theorem}[section]

\newtheorem{Definition}[Theorem]{Definition}

\newtheorem{Proposition}[Theorem]{Proposition}
\newtheorem{remarks}[Theorem]{Remarks}

\begin{document}

\title[Fractional Kirchhoff problem]{Fractional Kirchhoff problem with Hardy potential: Breaking of resonance}
\author[B. Abdellaoui, A. Azzouz, A. Bensedik]{B. Abdellaoui$^*$, A. Azzouz, A. Bensedik}
\thanks{$^*$ is  partially supported by project
MTM2016-80474-P, MINECO, Spain.} 

\address{\hbox{\parbox{5.7in}{\medskip\noindent {B. Abdellaoui, Laboratoire d'Analyse Nonlin\'eaire et Math\'ematiques
Appliqu\'ees. \hfill \break\indent D\'epartement de
Math\'ematiques, Universit\'e Abou Bakr Belka\"{\i}d, Tlemcen,
\hfill\break\indent Tlemcen 13000, Algeria.}}}}

\address{\hbox{\parbox{5.7in}{\medskip\noindent{A. Azzouz, Department of Mathematics, Faculty of Science,\\ University of Saida,\\
        BP. 138, 20000, Saida, Algeria. }}}}

\address{\hbox{\parbox{5.7in}{\medskip\noindent {A. Bensedik, Laboratoire Système dynamiques et applications. \hfill \break\indent D\'epartement de
Math\'ematiques, Universit\'e Abou Bakr Belka\"{\i}d, Tlemcen,
\hfill\break\indent Tlemcen 13000, Algeria.
\\[3pt]
        \em{E-mail addresses: }\\{\tt boumediene.abdellaoui@inv.uam.es, \tt abdelhalim.azzouz.cus@gmail.com, \tt a\_bensedik@mail.univ-tlemcen.dz        }.}}}}

\date{}

\begin{abstract}
In this paper we consider a fractional Kirchhoff problem with Hardy potential,
\begin{equation*}
\left\{
\begin{array}{rcll}
M(\dyle \iint_{\ren\times \ren} \frac{|u(x)-u(y)|^q}{|x-y|^{N+qs}}\,dx\,dy)(-\Delta) ^{s}u & = & \l\dfrac{u}{|x|^{2s}}+f(x,u) &\text{ in } \Omega, \\
u &> & 0 & \text{ in } \Omega,\\
u& = & 0 &\text{ in } \mathbb{R}^{N}\setminus \Omega,
\end{array}%
\right.
\end{equation*}
where $\Omega \subset \ren$ is a bounded domain containing the origin, $s\in (0, 1)$, $q\in (1,2]$ with $N>2s$, $\l>0$ and $f$ is a measurable nonnegative function with  suitable hypotheses.

The main goal of this work is to get the existence of solution for the largest class of $f$ without any restriction on $\l$. Our result covers also the local case $s=1$.
\end{abstract}

\maketitle

\section{Introduction}
The present work is concerned with existence of positive solutions for a class of fractional equation involving a Kirchhoff term and singular potential. More precisely we consider the problem
\begin{equation*}{(\mathcal{P}_{s,q})}
\left\{
\begin{array}{rcll}
M\bigg(\dyle \iint_{\ren\times \ren} \frac{|u(x)-u(y)|^q}{|x-y|^{N+qs}}\,dx\,dy\bigg)(-\Delta) ^{s}u & = & \l\dfrac{u}{|x|^{2s}}+f(x,u) & \text{ in } \Omega, \\
u & > & 0 & \text{ in } \Omega,\\
u & = & 0 & \text{ in } \mathbb{R}^{N}\setminus \Omega,
\end{array}%
\right.
\end{equation*}
where $\Omega \subset \ren$ is a bounded domain containing the origin, $s\in (0, 1)$, $q\in (1,2]$ with $N>2s$, $\l>0$ and $f$ is a measurable nonnegative function with  suitable hypotheses.
$M$  is a positive function defined on $\mathbb{R}^+$.
\newline
Here, for $0<s<1$, the fractional Laplacian
$(-\Delta)^s $  is defined by
\begin{equation}\label{fraccionario}
(-\Delta)^{s}u(x):=a_{N,s}\mbox{ P.V. }\int_{\mathbb{R}^{N}}{\frac{u(x)-u(y)}{|x-y|^{N+2s}}\, dy},
\end{equation}
where
$$a_{N,s}:=2^{2s-1}\pi^{-\frac N2}\frac{\Gamma(\frac{N+2s}{2})}{|\Gamma(-s)|}$$
is the normalization constant such that the identity
$$(-\Delta)^{s}u=\mathcal{F}^{-1}(|\xi|^{2s}\mathcal{F}u),\, \xi\in\mathbb{R}^{N} $$ holds for all $u\in \mathcal{S}(\mathbb{R}^N)$
where $\mathcal{F}u$ denotes the Fourier transform of $u$ and
$\mathcal{S}(\mathbb{R}^N)$ is the Schwartz class of tempered functions.

The problem $(\mathcal{P}_{s,q})$ is related to the following Hardy inequality, proved in \cite{B} and \cite{He},
\begin{equation}\label{hardy}
\frac{a_{N,s}}{2}\int\int_{D_\Omega}{\frac{|u(x)-u(y)|^2}{|x-y|^{N+2s}}}\,
dx\, dy\geq
\Lambda_{N,s}\int_{\Omega}{\frac{u^2}{|x|^{2s}}\,dx},\,u\in
\mathcal{C}^{\infty}_{0}(\ren),
\end{equation}
where $$
{D_\Omega} := \ren \times \ren \setminus \big( \mathcal{C} \Omega \times  \mathcal{C} \Omega \big) \,.
$$
and
\begin{equation}\label{cteHardy}
\Lambda_{N,s}=
2^{2s}\dfrac{\Gamma^2(\frac{N+2s}{4})}{\Gamma^2(\frac{N-2s}{4})}.
\end{equation}
Notice that $\L_{N,s}$ is optimal and not attained. \\

Nonlocal problems related to the Hardy potential were widely studied in the last years. Before staring our main result, let us begin by some previous results. In the case where $M\equiv 1$, then it is well known that problem $(\mathcal{P}_{s,q})$ has a solution if and only if $\l\le \L_{N,s}$. Moreover, for $\l\in (0,\L_{N,s}]$ fixed and taking into consideration the singular behavior near the origin of the Hardy potential, it holds that a natural condition is imposed on $f$  in order to get the existence of a solution. More precisely, if $f$ does not depends on $u$ ($f(x,u):=f(x))$, then $(\mathcal{P}_{s,q})$ has a solution if and only if $\dyle\io |x|^{-\g_{\l}}f(x)\,dx<\infty$ where $\g_{\l}$ will be explicitly stated in the next section. \\
If $f(x,u)=u_+^p$, then existence of solution holds if and only if $p<p_+(\l)<2s$. We refer for instance to \cite{AMPP} and \cite{BDT} for more details.

Where $M$ is not identically constant, then for the local case $s=1$ and without the Hardy potential, the problem can be expressed in the general form
\begin{equation*}{(\mathcal{P}_K)}
\quad-M(\displaystyle\int_\Omega \left|\nabla u \right|^2 dx)\Delta u=g(x, u) \quad\text{in} \quad \Omega, \quad u=0 \quad \text{on} \quad \partial\Omega.
\end{equation*}
$(\mathcal{P}_K)$ is well known as Kirchhoff type problem and is widely studied in the literature. The first motivation was the study of the stationary solution to the classical equation
$$\rho \dfrac{\partial^2u}{\partial t^2}-\left( \frac{p_0}{h}+\frac{E}{2L}\int_{0}^{L}\left|\dfrac{\partial u}{\partial x} \right|^2 dx \right)\dfrac{\partial^2 u}{\partial x^2} =0.$$
This equation is an extension of the classical d'Alembert's wave equation by considering the effects of the changes in the length of the string during the vibrations. In this equation, $L$  is the length of the string, $h$  is the area of the cross-section, $E$ is the Young modulus of the material, $\rho$ is the mass density and $P_0$ is the initial tension. The stationary multidimensional Kirchhoff equation is given by
\begin{equation*}{(\mathcal{P}_K)}
\quad-M(\displaystyle\int_\Omega \left|\nabla u \right|^2 dx)\Delta u=g(x, u) \quad\text{in} \quad \Omega, \quad u=0 \quad \text{on} \quad \partial\Omega,
\end{equation*}
which received much attention after the work of Lions \cite{lio}, where a functional setting was proposed to the problem.

The Kirchhoff function $M$ is assumed Lipschitz continuous, but not always monotone, even if the model proposed by Kirchhoff in 1883 $M(t)= a + bt^\beta$ where $a, b\ge 0$ and $\beta\ge0$ is monotone, see discussion in \cite{MRS} and \cite{PV}. If $M(0)=0$, the Kirchhoff problem is called degenerate, we may refer to \cite{BV} and references therein for more details on this case.

In \cite{Alcor}, using variational arguments the authors proved the existence of a positive solution of ${(\mathcal{P}_K)}$  if $M$ is a positive non-increasing function and $g$ satisfies the Ambrosetti-Rabinowitz condition. The case where $M$ is a non-decreasing function was considered in  \cite{azben}. In \cite{Ma}, the author collected some related problems to ${(\mathcal{P}_K)}$ and unified the presentation of the results. Especially, for $g$ such that $g(x,u)u\leq a \left|u \right|^2+b\left|u \right|$ ($a, b >0$) he showed both existence and uniqueness of solutions. In the case where $g(x,u)$ is asymptotically linear at infinity with respect to $u$, some existence results are proved in \cite{azben}. The singular case was considered in \cite{cor}.

The nonlocal case $s\in (0,1)$, was considered recently by several authors, including the critical and singular cases. We refer to \cite{MRS, PP, PXZ} and references therein for further details.

\

The novelty in our paper is to get some interactions between the Kirchhoff term and the Hardy potential in order to get the existence of non negative solution for a largest class of the datum $f$ and without any restriction on the parameter $\l$. We will show that, under some growth hypothesis on the Kirchhoff term, it is possible to break any resonance effect of the Hardy term. The resut covers also the local case $s=1$.

\

Through this paper we will assume that $M$ is a continuous function such that:
\newline (M1)  $M: \re^+\to \re^+$ is increasing.
\newline (M2)  $\exists m_0>0;\quad M(t)\geq m_0 \quad \forall t\in \re^+.$
\newline (M3) $\underset{t\rightarrow \infty }{\lim }M(t)=+\infty$

\

We begin by analyzing the case where $f$ does not depend on $u$, namely $f(x,u)=:=f(x)$. Then for $q=2$ and hence the problem has a variational structure, we are able to proof the next result.
\begin{Theorem}
Assume that $f \in L^{2}(\Omega)$ is such that $f\gneqq 0$, then for all $\lambda>0$, problem $(\mathcal{P}_{s,2})$ has a unique positive solution $\bar{u}\in H_{0}^{s}(\Omega)$.
\end{Theorem}
We mention here the strong regularity enjoyed by the solution. In the case where $M\equiv 1$, such a regularity is false including for $f\in L^\infty(\O)$ where $\l$ is closed to the Hardy constant $\L_{N,s}$.

In the general case $q<2$, we get the existence of a solution for a large class of the datum $f$. More precisely we have
\begin{Theorem}
Let $\O\subset \ren$ be a bounded domain such that $0\in \O$ and $1<q<\frac{N}{N-s}$. Suppose that $f \in L^{1}(\Omega)$ is such that $\dyle\int_\Omega f |x|^{-\theta} dx<c$ for some positive constant $\theta$. Then for all $\l>0$, the problem $(\mathcal{P}_{s,q})$ has a minimal solution $\bar{u}$ such that $\bar{u}\in W^{s,\s}_0(\Omega)$ for all $\s<\frac{N}{N-s}$.
  \end{Theorem}

In the second part of the paper we treat the case where $f(x,u)=u_+^p+\mu g$. In this case and under suitable hypotheses on $\mu$ and $g$, we will show the existence of a solution without any restriction on $\l$.

This paper is organized as follow. In Section \ref{prim}, we give some results related to fractional Sobolev spaces and some functional inequalities. The sense for which solutions are defined is also presented, namely energy solutions and weak solutions.\\

The proofs of the main existence results are given in Section \ref{proofs}. We begin by the case of variational structure and at the end of this section, we use approximating argument to treat the case of non variational structure.

\section{Preliminaries and functional setting }\label{prim}
Let $\O\subset \ren$, for $s\in (0,1)$ and $1\le p<+\infty$, the fractional Sobolev
spaces $W^{s,p}(\Omega)$, are defined by
$$
W^{s,p}(\Omega)\equiv
\Big\{ \phi\in
L^p(\O):\int_{\O}\int_{\O}\frac{|\phi(x)-\phi(y)|^p}{|x-y|^{N+ps}}dxdy<+\infty\Big\}
$$
$W^{s,p}(\O)$ is a Banach space endowed with the following norm
$$
\|\phi\|_{W^{s,p}(\O)}=
\Big(\dint_{\O}|\phi(x)|^pdx\Big)^{\frac 1p}
+\Big(\dint_{\O}\dint_{\O}\frac{|\phi(x)-\phi(y)|^p}{|x-y|^{N+ps}}dxdy\Big)^{\frac
1p}.
$$
We define the space $W_0^{s,p}(\Omega)$ as
$$
W_0^{s,p}(\Omega) := \left\{ u \in W^{s,p}(\ren): u = 0 \textup{ in } \ren\setminus \Omega \right\}.
$$
If $\O$ is a bounded regular domain, then using Poincaré type inequality, we can endow $W^{s,p}_{0}(\O)$ with the equivalent norm
$$
\|u\|_{W^{s,p}_0(\Omega)} := \left( \iint_{D_{\Omega}} \frac{|u(x)-u(y)|^p}{|x-y|^{N+sp}} dx dy \right)^{1/p},
$$
where
$$
D_{\Omega} := (\ren\times \ren) \setminus (\mathcal{C}\Omega \times \mathcal{C}\Omega) = (\Omega \times \ren) \cup (\mathcal{C}\Omega \times \Omega).$$

The next Sobolev inequality is proved in \cite{DPV}.
\begin{Theorem} \label{Sobolev}(Fractional Sobolev inequality)
Assume that $0<s<1, p>1$ in ordrer to get $ps<N$. Then, there exists a positive constant $S\equiv S(N,s,p)$ such that for all
$v\in C_{0}^{\infty}(\ren)$,
$$
\iint_{\re^{2N}}
\dfrac{|v(x)-v(y)|^{p}}{|x-y|^{N+ps}}\,dxdy\geq S
\Big(\dint_{\mathbb{R}^{N}}|v(x)|^{p_{s}^{*}}dx\Big)^{\frac{p}{p^{*}_{s}}},
$$
where $p^{*}_{s}= \dfrac{pN}{N-ps}$.
\end{Theorem}
In the particular case $p=2$, $H^{s}(\Omega):=W^{s,2}(\Omega)$ turns out to be a Hilbert spaces. If $\Omega=\mathbb{R}^N$, the Fourier transform provides yet an alternative definition, more precisely, the Plancherel identity leads to
$$
\int_{\mathbb{R}^N} |\xi|^{2s}|\mathcal{F}(u)(\xi)|^2d\xi={\frac{a_{N,s}}{2}}\int_{\mathbb{R}^N}\int_{\mathbb{R}^N}\frac{|u(x)-u(y)|^2}{|x-y|^{N+2s}}dxdy.
$$
The operator $(-\Delta)^s$  defined for
$u\in\mathcal{S}(\mathbb{R}^N)$ by $$(-\Delta)^s u :=a_{N,s}\mbox{ P.V. }\int_{\mathbb{R}^{N}}{\frac{u(x)-u(y)}{|x-y|^{N+2s}}\, dy},\, \,\qquad  s\in(0,1),$$ can be extended by density from $\mathcal{S}(\mathbb{R}^N)$ to $H^s(\mathbb{R}^N)$. In this way, the associated scalar product  can be reformulated as
\begin{equation*}\begin{split}
\langle u,v\rangle_{ H^s(\mathbb{R}^N)}&:=\langle (-\Delta)^s u, v\rangle+(u,v)\\
&:=
P.V.\iint_{\mathbb{R}^N\times \mathbb{R}^N}{\frac{(u(x)-u(y))(v(x)-v(y))}{|x-y|^{N+2s}}\,dx\,dy}+\int_{\mathbb{R}^N}uv\,dx.
\end{split}\end{equation*}

In the case of bounded domain if we denote by $H^{-s}(\Omega):=[H_0^s (\Omega)]^*$ the dual space
of $H_0^s (\Omega)$, then
$$(-\Delta)^s :H_0^s (\Omega)\rightarrow H^{-s}(\Omega),$$
is a continuous operator.

\

We state, by the next definition, the sense in which the solution is defined.
\begin{Definition}\label{defph4}
Assume that $h\in L^1(\O)$ and consider the problem
\begin{equation}\label{hho}
\left\{
\begin{array}{rcll}
(-\D)^{s} u &= & h & \text{ in }\Omega , \\
u & = & 0 & \text{  in }\ren\backslash \Omega , \\
\end{array}%
\right.
\end{equation}
we say that  $u\in L^1(\O)$ is a weak solution to \eqref{hho} if $u=0$ in $\ren\backslash \Omega$ and for all $\psi\in \mathbb{X}_s$, we have
$$
\io u((-\D)^{s}\psi) dx =\io h\psi dx.
$$
where
$$
\mathbb{X}_s\equiv \Big\{\psi\in \mathcal{C}(\ren)\,|\,\text{supp}(\psi)\subset \overline{\O},\,\, (-\Delta )^s\psi(x) \hbox{ pointwise defined and   } |(-\Delta )^s\psi(x)|<C \hbox{ in  } \O\Big\}.
$$
\end{Definition}
The next existence result is proved in \cite{LPPS}, \cite{CV1} and \cite{AAB}.
\begin{Theorem}\label{entropi}
Assume that $h\in L^1(\O)$, then problem \eqref{hho}  has a unique weak  solution $u$ obtained as the limit of the sequence $\{u_n\}_{n\in \mathbb{N}}$, where $u_n$ is the unique solution of the approximating problem
\begin{equation}\label{proOO}
\left\{\begin{array}{rcll} (-\Delta)^s u_n &= & h_n(x) & \mbox{  in  }\O,\\ u_n &= & 0 & \mbox{ in } \ren\backslash\O,
\end{array}
\right.
\end{equation}
with $h_n=T_n(h)$ and $T_n(\s)=\max(-n, \min(n,\s))$. Moreover,
\begin{equation} \label{tku}
T_k(u_n)\to T_k(u)\hbox{  strongly in }   H^{s}_{0}(\Omega), \quad \forall k > 0,
\end{equation}
\begin{equation} \label{L1u}
u \in L^\theta(\Omega) \,, \qquad  \forall  \ \theta\in \big[1, \frac{N}{N-2s}\big)\,
\end{equation}
 and
\begin{equation}\label{L1du}
\big|(-\Delta)^{\frac{s}{2}}   u\big| \in L^r(\Omega) \,, \qquad \forall  \  r \in \big[1,  \frac{N}{N-s} \big) \,.
\end{equation}
Furthermore
\begin{equation}\label{L1duu}
u_n\to u\mbox{  strongly in } W^{s,q_1}_0(\O)\mbox{  for all   }q_1<\frac{N}{N-s}.
\end{equation}
\end{Theorem}
As a subsequent of this existence result, we get a compactness result which will be systematically used in the sequel.

\begin{Theorem}\label{compa}
Assume that $1<q<\frac{N}{N-s}$ be fixed and consider the operator $D: L^1(\O)\to W^{s,q}_0(\O)$ defined by $D(h)=u$ where $u$ is the unique solution to the following problem
\begin{equation*}
\left\{\begin{array}{rcll} (-\Delta)^s u &= & h(x) & \mbox{  in  }\O,\\ u &= & 0 & \mbox{ in } \ren\backslash\O,
\end{array}
\right.
\end{equation*}
then $D$ is a compact operator, moreover, there exists a positive constant $C\equiv C(\O, s, N,q)$ such that
$$
||u||_{W^{s,q}_0(\O)}\le C||h||_{L^1(\O)}.
$$
\end{Theorem}

\

Since we are considering problem involoving the Hardy potential, we have tospecify the behavior of any solution in the neighborhood of the singular point $0$.

For $\l\le \L_{N,s}$ fixed, we define $\a>0$ by the identity
\begin{equation}\label{lambda}
\lambda=\lambda(\alpha)=\lambda(-\alpha)=\dfrac{2^{2s}\,\Gamma(\frac{N+2s+2\alpha}{4})\Gamma(\frac{N+2s-2\alpha}{4})}{\Gamma(\frac{N-2s+2\alpha}{4})\Gamma(\frac{N-2s-2\alpha}{4})}.
\end{equation}
Denote
\begin{equation}\label{g1}
\gamma:= \dfrac{N-2s}{2}-\alpha,
\end{equation}
then $0<\gamma\leq \dfrac{N-2s}{2}$ and $|x|^{-\gamma}$ is the unique solution to the equation
$$
(-\Delta)^{s} u=\l\frac{u}{|x|^{2s}},
$$
with $(-\Delta)^{s/2}(|x|^{-\g})\in L^2_{loc}(\R^N)$.
Form \cite{AMPP} we have the next result.
\begin{Theorem}
Consider the problem
\begin{equation}\label{jj}
\left\{
\begin{array}{rcll}
(-\Delta) ^{s}u & = & \l\dfrac{u}{|x|^{2s}}+f(x) &\text{ in } \Omega, \\
u &> & 0 & \text{ in } \Omega,\\
u& = & 0 &\text{ in } \mathbb{R}^{N}\setminus \Omega,
\end{array}%
\right.
\end{equation}
where $\Omega \subset \ren$ is a bounded domain containing the origin, $s\in (0, 1)$ with $N>2s$ and $0<\l\le \L_{N,s}$. Then problem \eqref{jj} has a nonnegative solution if and only if $\dyle\io f|x|^{-\g}dx <\infty$. In this case the unique solution of \eqref{jj} satisfies $u(x)\ge C|x|^{-\g}$ in $B_{r}(0)$.
\end{Theorem}

In what follows, $C$ denotes any positive constant that does not depend on the solution and that can change from line to other.

\section{Existence results in the case $f(x,u):=f(x)$.}\label{proofs}

In what follows and for the sake of simplicity we set
$$
||u||_{s,q}^q=\iint_{D_\O} \dfrac{|u(x)-u(y)|^{q}}{|x-y|^{N+qs}}\,dxdy,
$$
where
$$
D_{\Omega} := (\ren\times \ren) \setminus (\mathcal{C}\Omega \times \mathcal{C}\Omega).$$
Along this section we will assume that the function $M$ satisfies the conditions $(M1), (M2)$ stated in the introduction.

Recall that we are considering the next problem
\begin{equation*}{(\mathcal{P}_{s,2})}
\left\{
\begin{array}{rcll}
M(||u||_{s,2}^{2})(-\Delta) ^{s}u & = & \l\dfrac{u}{|x|^{2s}}+f(x) & \text{ in } \Omega, \\
u & > & 0  & \text{ in } \Omega,\\
u & = & 0 & \text{ in } \mathbb{R}^{N}\setminus \Omega.
\end{array}%
\right.
\end{equation*}
Our first existence result consists in the following.
\begin{Theorem}\label{thf1}
Assume that $f \in L^{2}(\Omega)$ is such that $f\gneqq 0$, then for all $\lambda>0$, the problem $(\mathcal{P}_{s,2})$ has a unique positive solution $\bar{u}$ in $H_{0}^{s}(\Omega)$.
\end{Theorem}

\begin{pf} Let $f \in L^{2}(\Omega)$ and fix $\l>0$, we will proceed by approximation.

Consider the truncated problem
\begin{equation*}{(\mathcal{P}^n_s)}
\left\{
\begin{array}{rcll}
M(||u_n||_{s,2}^{2})(-\Delta) ^{s}u_{n} & = & \dfrac{\lambda T_{n}(u_{n}^{+})}{|x|^{2s}+\frac 1n}+f(x) &\text{ in } \Omega, \\
u_{n} & > & 0  & \text{ in } \Omega,  \\
u_{n} & = & 0 & \text{ in } \mathbb{R}^{N}\setminus \Omega,
\end{array}
\right.
\end{equation*}
It is clear that solutions of problem $(\mathcal{P}^n_s)$ are critical points of the functional $J_n$ given by
$$J_{n}(v)=\dfrac{1}{2}\widehat{M}(\left\Vert v\right\Vert_s^{2}) -\lambda \int_\Omega \dfrac{\widehat{T}_{n}(v^{+})}{\left\vert x\right\vert ^{2s}+\frac 1n}dx -\int_\Omega fvdx, $$
where $$ \widehat{M}(t)=\int^t_0 M(\tau)d\tau, \qquad \widehat{T}_{n}(t)=\int^t_0 T_n(\tau)d\tau. $$
Taking into consideration the behavior of $M$ and setting
$$
D_n=\inf_{H_0^s(\Omega)\backslash \{0\}}J_n(u),
$$
we reach that $D_n$ is achieved and then we get the existence of $u_{n}\in H_0^s(\Omega)$ that solves $(\mathcal{P}^n_s)$.
We claim that the sequence $\{u_n\}_n$ is bounded in $H_0^s(\Omega)$. We argue by contradiction. Assume that $||u_n||_{s,2}\to \infty$ as $n\to \infty$.
Taking $u_n$ as a test function in $(\mathcal{P}^n_s)$ and using Hardy and H\"{o}lder inequalities, it follow that
\begin{equation}\label{sati1}
\begin{array}{ll}
M(||u_{n}||_{s,2}^{2})||u_{n}||_{s,2}^{2}\leq
\lambda \displaystyle \int_\Omega \dfrac{\left\vert u_{n}\right\vert ^{2}}{\left\vert
x\right\vert ^{2s}}dx+\displaystyle \int_\Omega fu_ndx \\
 \leq \dfrac{\lambda }{\Lambda _{N,s}}\left\Vert u_{n}\right\Vert^2
_{s,2}+\left\Vert f\right\Vert_{L^2 (\Omega)}||u||_{s,2}.
\end{array}
\end{equation}
Thus
\begin{equation}
||u||^2_{s,2} \left( M\left(||u||^2_{s,2}\right)-\frac{\lambda}{\Lambda_{N,s}}\right) \leq 2^{-1}\left(||u||^2_{s,2}+ \left\Vert f\right\Vert ^{2}_{L^2 (\Omega)}\right),
\end{equation}
from which we have,
\begin{equation}\label{esti0}
2||u_{n}||^2_{s,2}( M(||u||^2_{s,2}) -\frac{\lambda}{\Lambda_{N,s}}-1) \leq \left\Vert f\right\Vert ^{2}_{L^2 (\Omega)}.
\end{equation}
By $(M2)$, we deduce that
$$
 M(||u_{n}||^{2}_{s,2}) -\frac{\lambda}{\Lambda_{N,s}}-1\to \infty\mbox{  as  }n\to \infty.
 $$
Hence we reach a contradiction with \eqref{esti0}. Thus we conclude that $\{u_n\}_n$ is bounded in $H_0^s(\Omega)$ and the claim follows.

Therefore, we get the existence of $\bar{u}\in H_0^s(\Omega)$ such that, up to a subsequence, $u_n\rightharpoonup \bar{u}$ weakly in $H_0^s(\Omega)$, and $u_n\longrightarrow \bar{u}$ strongly in $L^q(\Omega)$ for all $1\leq q <2^*_s$.

We claim that there exists a positive constant $\rho$ such that up to a subsequence denoted by $u_n$, we have
\begin{equation}\label{gamma}
\frac{\lambda}{M(||u_{n}||_{s,2}^{2})}<\Lambda_{N,s}-\rho \mbox{  for all  }n.
\end{equation}
By virtue of $(M1)$ and since $||u||_{s,2}\leq C$,  we obtain,
\begin{equation*}
\liminf_{n\rightarrow \infty} M(||u||_{s,2}^{2}):=\mu>0.
\end{equation*}
Therefore, we get the existence of a subsequence denoted by $\{u_{n}\}_n$ such that,
$$M(||u_n||_{s,2}^{2})\rightarrow \mu\mbox{  as  }n\to \infty. $$
Obviously, $u_{n}\rightharpoonup \bar{u}$ weakly in $H_0^s(\Omega)$, thus by passing to the limit in the problem of $u_{n}$, it follows that $\bar{u}$ solves
\begin{equation*}
\mu (-\Delta)^{s}\bar{u} = \l \frac{\bar{u}}{\left|x\right|^{2s}}+f.
\end{equation*}
Thus $\bar{u}\in H_0^s(\Omega)$, satisfies $\bar{u}\ge 0$ in $\O$ and
$$
(-\Delta) ^{s}\bar{u}=\dfrac{\lambda}{\mu}\dfrac{ \bar{u}}{|x|^{2s}}+\dfrac{f(x)}{\mu} \text{ in } \Omega.
$$
From the result of \cite{AMPP}, we conclude that $\lambda/\mu\le \Lambda_{N,s}$.

\

Let us show that $\lambda/\mu<\Lambda_{N,s}-\rho$ for some positive constant $\rho>0$.

If $\lambda/\mu = \Lambda_{N,s}$, then $\bar{u}$ is solution of the equation
\begin{equation*}
	(-\Delta) ^{s}\bar{u} =\Lambda_{N,s}\dfrac{ \bar{u}}{
		\left\vert x\right\vert ^{2s}}+\dfrac{f(x)}{\mu} \quad \text{in} \hspace{0,15cm} \Omega,
\end{equation*}
Using again the result of \cite{AMPP}, one can see that $\bar{u}(x)\ge C|x|^{-\frac{N-2s}{2}}$ in a small ball $B_\eta(0)$. Thus
$$
\int_{B_\eta(0)}\bar{u}^{2^*_s}dx\ge C\int_{B_\eta(0)}|x|^{-N} dx=\infty,
$$
which is a contradiction with the fact that $\bar{u}\in H_0^s(\Omega)$. Hence we get the existence of a positive constant $\rho$ such that
\begin{equation}\label{estim1}
\frac{\lambda}{M(||u_{n}||_{s,2}^2)}<\Lambda_{N,s}-\rho.
\end{equation}
To conclude, we prove the strong convergence of the sequence $\{u_n\}_n$ in $H^s_0(\Omega)$.\\
It is clear that $u_n$ solves
	$$(-\Delta) ^{s}u_{n}  =\dfrac{\lambda}{M(||u_n||_{s,2}^{2})}\dfrac{ u_{n}}{\left( \left\vert x\right\vert +n^{-1}\right) ^{2s}}+\dfrac{f(x)}{M(||u_n||_{s,2}^{2})}.$$
	Using $u_n-\bar{u}$ as a test function, we obtain that
	\begin{equation}\label{estim01}
	\displaystyle \int_\Omega (-\Delta) ^{s}u_{n}(u_{n}-\bar{u})dx=
	\dfrac{\lambda}{M(||u_n||_{s,2}^{2})}\displaystyle \int_\Omega \dfrac{ u_{n}(u_n-\bar{u})}{\left( \left\vert x\right\vert +n^{-1}\right) ^{2s}}dx+\dfrac{1}{M(||u_n||_{s,2}^{2})}\displaystyle \int_\Omega f(x)(u_n-\bar{u}) dx.
	\end{equation}
Since the term $\dfrac{1}{M(||u_n||_{s,2}^2)}$ is bounded away from zero and $u_n-\bar{u}\to 0$ strongly in $L^2(\Omega)$, it holds that $$\dfrac{1}{M(||u_n||_{s,2}^2)}\int_\Omega f(x)(u_n-\bar{u}) dx\to 0\mbox{  as  }n\to \infty.$$
We deal now with the term $\displaystyle\int_\Omega\dfrac{ u_{n}\left( u_n-\bar{u}\right) }{\left( \left\vert x\right\vert+n^{-1}\right)^{2s} }dx$.

We have
\begin{eqnarray*}
\int_\Omega\dfrac{u_{n}( u_n-\bar{u})}{(|x|+n^{-1})^{2s}}dx &= & \displaystyle\int_\Omega\dfrac{ \left( u_n-\bar{u}\right)^2 }{\left( \left\vert x\right\vert+n^{-1}\right)^{2s} }dx+\displaystyle\int_\Omega\dfrac{ \bar{u}\left( u_n-\bar{u}\right) }{\left( \left\vert x\right\vert+n^{-1}\right)^{2s} }dx\\
&=& \displaystyle\int_\Omega\dfrac{ \left( u_n-\bar{u}\right)^2 }{\left( \left\vert x\right\vert+n^{-1}\right)^{2s} }dx+\displaystyle\int_\Omega\dfrac{ u_n-\bar{u} }{\left( \left\vert x\right\vert+n^{-1}\right)^s }\dfrac{\bar{u}}{\left( \left\vert x\right\vert+n^{-1}\right)^s}dx.
\end{eqnarray*}
Taking into consideration that $\dfrac{\bar{u}}{{(|x|+n^{-1})^s}}\to \dfrac{\bar{u}}{|x|^s}$ strongly in $L^2(\O)$ and that
$\dfrac{u_n-\bar{u}}{{\left( \left\vert x\right\vert+n^{-1}\right)^s}} \rightharpoonup 0$ weakly in $L^2\left(\Omega\right),$ (as $ n\to \infty$), we obtain
$$
\displaystyle\int_\Omega\dfrac{ u_n-\bar{u} }{\left( \left\vert x\right\vert+n^{-1}\right)^s }\dfrac{\bar{u}}{\left( \left\vert x\right\vert+n^{-1}\right)^s}dx\to 0\mbox{  as  }n\to \infty.
$$
Consequently, and using again the Hardy inequality, we get
\begin{equation}\label{estim02}
\int_\Omega\dfrac{ u_{n}\left( u_n-\bar{u}\right) }{\left( \left\vert x\right\vert+n^{-1}\right)^{2s} }dx=\int_\Omega\dfrac{ \left( u_n-\bar{u}\right)^2 }{\left( \left\vert x\right\vert+n^{-1}\right)^{2s} }dx+o(1)
\leq\dfrac{1}{\Lambda_{N,s}} \left\Vert u_{n}-\bar{u}\right\Vert_{s,2}^{2}+o(1).
\end{equation}
Now, by the weak convergence, it holds that
$$\displaystyle \int_\Omega (-\Delta) ^{s}u_{n}(u_{n}-\bar{u})dx=\left\Vert u_{n}-\bar{u}\right\Vert_{s,2}^{2}+o(1).$$
Therefore, by estimates \eqref{estim01} and \eqref{estim02}, it follows that
$$\left\Vert u_{n}-\bar{u}\right\Vert_{s,2}^{2}\leq\dfrac{\lambda}{\Lambda_{N,s} M(||u_n||_{s,2}^2)}\left\Vert u_{n}-\bar{u}\right\Vert_{s,2}^{2}+o(1).$$
Recall that by \eqref{estim1}, we have $\dfrac{\lambda}{M(||u_n||_{s,2}^2)}<\Lambda_{N,s}-\rho$ for all $n$. Thus
$$\left\Vert u_{n}-\bar{u}\right\Vert_{s,2}^{2}<\displaystyle \dfrac{\Lambda_{N,s}-\rho}{\Lambda_{N,s} }\left\Vert u_{n}-\bar{u}\right\Vert_{s,2}^{2}+o(1).$$
Since $0<\displaystyle\dfrac{\Lambda_{N,s}-\rho}{\Lambda_{N,s} }<1$, then $\left\Vert u_{n}-\bar{u}\right\Vert_{s,2}^{2}\longrightarrow 0$ as $n\to \infty$ and then  $u_n\longrightarrow \bar{u}$ strongly in $H^s_0(\Omega)$.

As a consequence it follows that $M(||u_n||_{s,2})\to M(||u||_{s,2})$. Passing to the limit in problem $(\mathcal{P}^n_s)$, we obtain that $\bar{u}$ solves $(\mathcal{P}_{s,2})$.

It remains to proof yet the uniqueness part.

Assume that $u,v$ are two positive solutions to problem $(\mathcal{P}_{s,2})$. To show the uniqueness we have just to prove that  $M(||u||_{s,2}^{2})=M(||v||_{s,2}^{2})$.
Assume by contradiction that $M(||u||_{s,2}^{2})<M(||v||_{s,2}^{2})$. Setting $u_1=M(||u||_{s,2}^{2})u$ and $v_1=M(||v||_{s,2}^{2})v$, it holds that
$$
(-\Delta) ^{s}v_1=\frac{\l}{M(||v||_{s,2}^{2})} \,\dfrac{v_1}{|x|^{2s}}+f(x),$$
and
$$
(-\Delta) ^{s}u_1=\frac{\l}{M(||u||_{s,2}^{2})} \,\dfrac{u_1}{|x|^{2s}}+f(x).
$$
Since $\frac{\l}{M(||v||_{s,2}^{2})}<\frac{\l}{M(||u||_{s,2}^{2})}$ and $u_1=v_1=0$ in $\mathbb{R}^{N}\setminus \Omega$, then $v_1<u_1$ in $\O$. Thus
$$
M(||v||_{s,2}^{2})v<M(||u||_{s,2}^{2})u \mbox{  in  }\Omega.
$$
Recall that, by hypothesis, $M(||u||_{s,2}^{2})<M(||v||_{s,2}^{2})$, thus $v<u$ in $\O$. Going back to the equation of $v$ and using $v$ as a test function, it follows that
$$
M(||v||_{s,2}^{2})||v||_{s,2}^{2}=\lambda\io\dfrac{v^2}{|x|^{2s}}dx+\io fvdx \le \lambda\io\dfrac{u^2}{|x|^{2s}}dx+\io fudx= M(||u||_{s,2}^{2})||u||_{s,2}^{2}
$$
Since $M(t^2)t^2$ is a strict monotone function, we conclude that $||v||_{s,2}^{2}<||u||_{s,2}^{2}$ and then $M(||v||_{s,2}^{2})\le M(||u||_{s,2}^{2})$, a contradiction with the main hypothesis. Hence the uniqueness result follows.

\end{pf}

\begin{remarks}\label{rem}
It is not difficult to show that the above existence and the uniqueness result holds also for all $f\in L^{\frac{2N}{N+2s}}(\O)$.
\end{remarks}

As a consequence of the previous construction, we have the next regularity result. 
\begin{Proposition}\label{rrem}
Assume that $f\in L^{\frac{2N}{N+2s}}(\O)$, then for all $\s<2_s^*$, there exists a positive constant $C:=C(N, \O, \s)$ such that if $u$ is the solution to problem $(\mathcal{P}_{s,2})$, we have 
$$
||u||_{L^{\s}(\O)}\le C||f||_{L^{\frac{2N}{N+2s}}(\O)}.
$$
\end{Proposition}
\begin{proof}
Since $u$ solves $(\mathcal{P}_{s,2})$, then $\frac{\l}{M(||u||_{s,2}^{2})}\le \L_{N,s}$ and
$$
(-\Delta) ^{s}u=\frac{\l}{M(||u||_{s,2}^{2})} \,\dfrac{u}{|x|^{2s}}+\frac{f(x)}{M(||u||_{s,2}^{2})}\le
\L_{N,s}\dfrac{u}{|x|^{2s}}+\frac{1}{C}f(x).
$$
Setting $w$ to be the unique solution to the problem
$$
\left\{
\begin{array}{rcll}
(-\Delta) ^{s}w &= & \L_{N,s}\dfrac{w}{|x|^{2s}}+\frac{1}{C}f(x) & \text{ in } \Omega, \\
w & = & 0 &\text{ in } \mathbb{R}^{N}\setminus \Omega,
\end{array}%
\right.
$$
then $u\le w$ in $\O$. Notice that from \cite{AMPP}  we obtain that, for all $\s<2^*_s$,
$$
||w||_{L^{\s}(\O)}\le C||\frac{1}{C}f||_{L^{\frac{2N}{N+2s}}(\O)}\le C||f||_{L^{\frac{2N}{N+2s}}(\O)}.
$$
Hence we conclude.
\end{proof}

To complete this part, we have the next compactness result.
\begin{Proposition}\label{compannn}
Let $\{f_n\}_n\subset L^{\frac{2N}{N+2s}}(\O)$ be such that $f_n\ge 0$ and $||f_n||_{L^{\frac{2N}{N+2s}}(\O)}\le C$ for all $n$. Define $u_n$, the unique solution to problem $(\mathcal{P}_{s,2})$ with $f:=f_n$. Then there exists a $u\in H^s_0(\O)$, such that, us to a subsequence, $u_n\to u$ strongly in $H^s_0(\O)$.
\end{Proposition}
\begin{proof}
Let $w_n$ be the unique solution to the problem
$$
\left\{
\begin{array}{rcll}
(-\Delta) ^{s}w_n &= & \L_{N,s}\dfrac{w_n}{|x|^{2s}}+\frac{1}{C}f_n(x) & \text{ in } \Omega, \\
w_n & = & 0 &\text{ in } \mathbb{R}^{N}\setminus \Omega,
\end{array}%
\right.
$$
then $u_n\le w_n$ and  for all $\s<2^*_s$,
$$
||w_n||_{L^{\s}(\O)}\le C||f_n||_{L^{\frac{2N}{N+2s}}(\O)}.
$$
Hence $||u_n||_{L^{\s}(\O)}\le C$ for all $n$.
If $M(||u_n||_{s,2}^{2})\le 2\frac{\l}{\L_{N,s}}$, then $||u_n||_{s,2}^{2}\le C$. Moreover, if $M(||u_n||_{s,2}^{2})\ge 2\frac{\l}{\L_{N,s}}$, then using $u_n$ as a test function in the problem solved by $u_n$, it follows that
$$
M(||u_{n}||_{s,2}^{2})||u_{n}||_{s,2}^{2}\leq
\frac{\lambda}{\L_{N,s}} ||u_{n}||_{s,2}^{2}+\int_\Omega f_nu_ndx.
$$
Thus
$$
(M(||u_{n}||_{s,2}^{2})-\frac{\lambda}{\L_{N,s}}) ||u_{n}||_{s,2}^{2}
 \leq \int_\Omega f_nu_ndx.
$$
Hence, using H\"older and Sobolev inequalities, we conclude that
$$
||u_{n}||_{s,2}^{2}\le C\mbox{  for all }n.
$$
Hence $\{u_n\}_n$ is bounded in $H^s_0(\O)$. Therefore, we get the existence of $u\in H^s_0(\O)$ such that, up to a subsequence, $u_n\rightharpoonup u$ weakly in $H^s_0(\O)$, $u_n\to u$ strongly in $L^\s(\O)$ for all $\s<2^*_s$ and $(M(||u_{n}||_{s,2}^{2}):=\a_n\to \a_0\ge C$. Notice that $u_n$ solves
	 \begin{equation*}
	 \left\{
\begin{array}{rcll}
( -\Delta)^s u_n &=& \frac{\l}{ M(||v||_{s,q}^q)}\frac{u_n}{|x|^{2s}}+\dfrac{f_n}{ M(||u_n||_{s,q}^q)} &\text{ in }\Omega\\
	u_n &= & 0 &\text{ in } \mathbb{R}^{N}\setminus \Omega.
	 \end{array}
\right.
	 \end{equation*}
Setting $h_n(x):= \frac{\l}{ M(||v||_{s,q}^q)}\frac{u_n}{|x|^{2s}}+\dfrac{f_n}{ M(||u_n||_{s,q}^q)}$, it holds that $||h_n||_{L^1(\O)}\le C$. Therefore by the compactness result in \cite{AAB}, we reach that, up to a subsequence, $u_n\to u$ strongly in $W^{s,q}_0(\O)$ for all $q<\frac{N}{N-s}$.
Thus by interpolation, it holds that $u_n\to u$ strongly in $W^{s,q}_0(\O)$ for all $q<2$ and the result follows.
\end{proof}

\subsection{Some results related to non variational problems.}\label{non}

In this subsection we will assume that $q<2$. Consider the problem
\begin{equation}\label{eq:auxi1}
\left\{
\begin{array}{rcll}
M(\left\Vert u\right\Vert_{s,q} ^{q})(-\Delta) ^{s}u & = & f & \text{ in } \Omega, \\
u & = & 0 &\text{ in } \mathbb{R}^{N}\setminus \Omega,
\end{array}%
\right.
\end{equation}
It is clear that problem \eqref{eq:auxi1} can not be viewed and treated in a variational setting.
We begin by proving the next proposition.
\begin{Proposition}\label{th:auxi}
Assume that $f \in L^{2}(\Omega)$ and $1<q<2$, then problem \eqref{eq:auxi1} has a unique solution $u_{q}$ in $H^s_0(\Omega)$.
\end{Proposition}
\begin{pf} To show the existence of a solution we will use the Schauder's Fixed Point Theorem.
Let $E=W_{0}^{s,q}(\Omega)$ and consider the operator $T: E\longrightarrow E$ defined by $T(v)=u$ where $u$ is the unique solution for the problem
	 \begin{equation*}
	 \left\{
\begin{array}{rcll}
( -\Delta)^s u &=& \dfrac{f}{ M(||v||_{s,q}^q)} &\text{ in }\Omega\\
	u &= & 0 &\text{ in } \mathbb{R}^{N}\setminus \Omega.
	 \end{array}
\right.
	 \end{equation*}
Taking into consideration the properties of $M$ and classical results on regularity of the fractional Laplacian we conclude that $T$ is well defined
and
\begin{equation}\label{YY}
||u||_{H^s_0(\Omega)}\leq \dfrac{c_0}{M(||v||_{s,q}^q)}\|f\|_{L^2(\Omega)}\leq \dfrac{c_0}{m_0}\|f\|_{L^2(\Omega)}.
	 \end{equation}
Hence
	 \begin{equation*}
||u||_E \leq \dfrac{c_1}{m_0} \left\|f \right\|_{L^2(\Omega)}
	 \end{equation*}
	 for some positive constant $c_1$ depending only on $\Omega, N$ and $s$.
Choosing $r=m^{-1}_0c_1||f||_{L^2(\Omega)}$, it holds that $T(B_r(0))\subset B_r(0)$, where $ B_r(0)$ is the ball centered at zero with radius $r>0$ of $E$.\\
It is not difficult to show that $T$ is continuous. To finish we have just to prove that $T$ is compact.

\

Let $\{v_n\}_n$ be a bounded sequence in $E$. Define $u_n=T(v_n)$, then from \eqref{YY} it holds that $\{u_n\}_n$ is bounded in $H^s_0(\O)$ and a priori in $E$. Hence, up to a subsequence, we get the existence of $\hat{u}\in H^s_0(\O)$ such that $u_n\rightharpoonup \hat{u}$ weakly in $H^s_0(\O)$.

Setting $g_n=\dfrac{f}{ M(\left\Vert v_n\right\Vert_{s,q} ^{q})}$, then $||g_n||_{L^1(\O)}\le C$ for all $n$. Hence from the compactness result in Theorem \ref{compa} it holds that $u_n\to \hat{u}$ strongly in $W^{s,p}_0(\O)$ for all $p<\frac{N}{N-s}$. Taking into consideration that the sequence $\{u_n\}_n$ is bounded in $H^s_0(\O)$ and using Vitali's Lemma, we conclude that  $u_n\to \hat{u}$ strongly in $W^{s,\sigma}_0(\O)$ for all $\sigma<2$, in particular $u_n\to \hat{u}$ strongly in $E$. Hence $T$ is compact.

Thus by Schauder's Fixed Point Theorem, we get the existence of $\tilde{u}\in E$ such that $T(\tilde{u})=\tilde{u}$ and then $\tilde{u}$ solves problem \eqref{eq:auxi1}.
It is not difficult to show that $u_q\in H^s_0(\O)$.

We prove now the uniqueness part. Suppose that $u_1$ and $u_2$ are two solutions to the problem \eqref{eq:auxi1}, then
	 $$M(\left\Vert u_1\right\Vert_{s,q} ^{q})\left( -\Delta\right)^s u_1=f=M(\left\Vert u_2\right\Vert_{s,q} ^{q})\left( -\Delta\right)^s u_2.$$
Thus
$$
(-\Delta)^s\bigg(M(||u_1||_{s,q}^{q})u_1-M(||u_2||_{s,q}^{q})u_2\bigg)=0.
$$
Since $M(||u_i||_{s,q}^{q})u_i\in H^s_0(\O)$ for $i=1,2$, then
 \begin{equation}\label{eq:norme}
 M(\left\Vert u_1\right\Vert_{s,q} ^{q})u_1=M(\left\Vert u_2\right\Vert_{s,q} ^{q})u_2.
	 \end{equation}
Therefore we get $M(||u_1||_{s,q} ^{q})||u_1||_{s,q}=M(||u_2||_{s,q} ^{q})||u_2||_{s,q}$. Recall that $q>1$, then by $(M1)$ and $(M2)$, the function $t\longmapsto M(t^q)t$ is increasing and then $||u_1||_{s,q}=||u_2||_{s,q}$. Going back to \eqref{eq:norme}, we obtain $u_1=u_2$.
	  \end{pf}
	
In the case where $1<q<\frac{N}{N-s}$, we are able to show that problem \eqref{eq:auxi1} has a unique solution for all $f\in L^1(\O)$.
\begin{Theorem}
Assume that $f \in L^{1}(\Omega)$ and $1<q<\frac{N}{N-s}$, then problem \eqref{eq:auxi1} has a unique solution $\bar{u}$ in $W_{0}^{s,\sigma}(\Omega)$ for all $\sigma<\frac{N}{N-s}$. In particular $\bar{u}\in W_{0}^{s,q}(\Omega)$. Moreover defining $T:L^1(\O)\to W^{s,q}_0(\O)$ with $u=T(f)$ being the unique solution to problem \eqref{eq:auxi1}, then $T$ is a compact operator.
\end{Theorem}
\begin{pf}
We set $f_n=T_n(f)$, then $f_n\in L^\infty(\Omega)$. By Theorem \ref{th:auxi}, it follows that the problem
$$
{(\mathcal{P}^n_{s,q})}
\left\{
\begin{array}{rcll}
M(\left\Vert u_n\right\Vert_{s,q} ^{q})(-\Delta) ^{s}u_n & = & f_n(x) & \text{  in   } \Omega, \\
u_n & = & 0 & \text{  in   } \mathbb{R}^{N}\setminus \Omega,
	\end{array}
	\right.
$$
has a unique solution $u_n$. We claim that $M(\left\Vert u_n\right\Vert_{s,q} ^{q})<C$ for all $n\geq 1$.

We argue by contradiction. Assume that $M(\left\Vert u_n\right\Vert_{s,q} ^{q})\longrightarrow +\infty$ as $n\to \infty$. Then by the properties of $M$, we reach that $\lim\limits_{n\to \infty}\left\Vert u_n\right\Vert_{s,q}=\infty$.

Define $g_n\equiv \dfrac{f_n}{M(\left\Vert u_n\right\Vert_{s,q} ^{q})}$, then $|g_n|\le \dfrac{|f_n|}{m_0}$ and hence $||g_n||_{L^1(\O)}\le C$. Thus using Theorem \ref{compa} it holds that $||u_n||_{W^{s,\sigma}_0(\Omega)}\leq \frac{C}{m_0}||f_n||_{L^1(\O)}\le C$ for all $n$ and for all $\sigma <\frac{N}{N-s}$. Choosing $\s=q$, we get $||u_n||_{W^{s,q}_0(\Omega)}\leq C$, a contradiction with the main hypothesis. Thus $M(\left\Vert u_n\right\Vert_{s,q} ^{q})<C$ for all $n$ and then claim follows.

Consequently $\left\Vert u_n\right\Vert_{s,q}\leq C$, therefore there exists $\bar{u}\in W^{s,q}_0(\Omega)$ such that, up to a subsequence,  $u_n\rightharpoonup \bar{u}$ weakly in $W^{s,q}_0(\Omega)$.

Let us show that $u_n\longrightarrow\bar{u}$ strongly in $W^{s,q}_0(\Omega)$.

Recall that $$(-\Delta)^s u_n=\dfrac{f_n}{M(\left\Vert u_n\right\Vert_{s,q} ^{q})}, $$ by the boundedness of the sequence $\{\dfrac{f_n}{M(\left\Vert u_n\right\Vert_{s,q} ^{q})}\}_n$ in $L^1(\O)$ and the compactness result in Theorem \ref{compa}, we conclude that
$u_n\longrightarrow \bar{u}$ strongly in $W^{s,\sigma}_0(\Omega)$, for all $\sigma<\dfrac{N}{N-s}$ and in particular $u_n\longrightarrow \bar{u}$ strongly in $ W^{s,q}_0(\Omega)$

Hence $M(\left\Vert u_n\right\Vert_{s,q} ^{q})\longrightarrow M(\left\Vert \bar{u}\right\Vert_{s,q} ^{q})$.
It is clear that $\bar{u}$ solves the problem \eqref{eq:auxi1}. The uniqueness follows as in the proof of Theorem \ref{th:auxi}.

Define now $T:L^1(\O)\to W^{s,q}_0(\O)$ where $u=T(f)$ is the unique solution to problem \eqref{eq:auxi1}. It is clear that $T$ is well defined. To show that $T$ is compact, we consider a sequence $\{f_n\}_n\subset L^1(\O)$ and define $u_n=T(f_n)$. Then
$$(-\Delta)^s u_n=\dfrac{f_n}{M(\left\Vert u_n\right\Vert_{s,q} ^{q})}. $$ Now, repeating the same computations as above and using the compactness result in Theorem \ref{compa} we deduce that, up to a subsequence, $u_n\longrightarrow \bar{u}$ strongly in $W^{s,\sigma}_0(\Omega)$, for all $\sigma<\frac{N}{N-s}$ and in particular $u_n\longrightarrow \bar{u}$ strongly in $ W^{s,q}_0(\Omega)$. Hence the result follows.
\end{pf}

\

We are now able to state the main result of this subsection.
\begin{Theorem}\label{generalq}
Let $\O\subset \ren$ be a bounded domain such that $0\in \O$ and $1<q<\frac{N}{N-s}$. Suppose that $f \in L^{1}(\Omega)$ is such that $\displaystyle{\int_\Omega f \left| x\right| ^{-\theta} dx}<c$ for some positive constant $\theta$. Then for all $\l>0$, the problem
$$
(\mathcal{P}^{\lambda}_{s,q})
\left\{
\begin{array}{rcll}
M(\left\Vert u\right\Vert_{s,q}^{q})(-\Delta) ^{s}u & = & \l\dfrac{u}{\left| x\right|^{2s} }+f &\text{ in } \Omega, \\
u & = & 0 &\text{ in } \mathbb{R}^{N}\setminus \Omega,
\end{array}
\right.
$$
has a solution $u$ such that $\bar{u}\in W^{s,\s}_0(\Omega)$ for all $\s<\frac{N}{N-s}$.
  \end{Theorem}
\begin{pf}
Let $\lambda>0$ be fixed and define $f_n=T_n(f)$. Consider the approximating problem
\begin{equation*}{(\mathcal{P}^{n}_{s,q})}
\left\{
\begin{array}{rcll}
M(||u_n||^q_{s,q})(-\Delta) ^{s}u_n & = & \l\dfrac{T_n(u_n)}{\left| x\right|^{2s}+\frac 1n}+f_n(x) & \text{ in } \hspace{0,15cm} \Omega, \\
u_n & = & 0 & \text{ in } \hspace{0,15cm}  \mathbb{R}^{N}\setminus \Omega.
\end{array}
\right.
\end{equation*}	
Using a direct variation of Theorem \ref{compa}, we can show that Problem $({P}^{n}_{s,q})$ has a solution $u_n\in H^s_0(\O)$.

Let begin by proving that $M(||u||^q_{s,q})<m_1$.
We argue by contradiction, if $M(||u_n||^q_{s,q})\longrightarrow +\infty$, then $||u||^q_{s,q}\longrightarrow +\infty$. Observe that
\begin{equation}\label{tlemcen}
 (-\Delta) ^{s}u_n=\dfrac{\lambda}{M(||u_n||^q_{s,q})}\,\dfrac{T_n(u_n)}{\left| x\right|^{2s}+\frac 1n}+\dfrac{f_n}{M(||u_n||^q_{s,q})}
\end{equation}
For simplicity of typing we set
\begin{equation}\label{tr}
g_n\equiv\dfrac{\lambda}{M(||u_n||^q_{s,q})}\,\dfrac{T_n(u_n)}{|x|^{2s}+\frac 1n}+\dfrac{f_n}{M(||u_n||^q_{s,q})}.
\end{equation}
We claim that $\bigg\|\dfrac{u_n}{|x|^{2s}}\bigg\|_{L^1(\Omega)}<C$ for some positive constant $C$.

To prove the claim we define $\varphi$ as the unique solution to the problem
\begin{equation}\label{eps}
 \left\{
\begin{array}{rcll}
(-\Delta) ^{s}\varphi & = & \beta\dfrac{\varphi}{\left| x\right|^{2s} }+1 &\text{ in }\Omega, \\
\varphi & =  & 0 & \text{ in }\mathbb{R}^{N}\setminus \Omega,
\end{array}
\right.
\end{equation}	
where $0<\beta<\L_{N,s}$ to be chosen later. Notice that $\varphi\in L^\infty(\O\backslash B_r(0))$ and $\varphi\approx |x|^{-\g_\beta}$ where
\begin{equation}\label{bb}
\gamma_\beta:= \dfrac{N-2s}{2}-\alpha,
\end{equation}
and $\a$ is given by \eqref{lambda} where $\l$ is substituted by $\beta$. Il it not difficult to show that
$\g_\beta\to 0$ as $\beta\to 0$.

Using $u_n$ as a test function in \eqref{eps} and integrating over $\Omega$, it holds that
$$
\beta\int_\Omega \dfrac{\varphi u_n}{\left| x\right|^{2s}}dx+\int_\Omega u_n dx=\dfrac{\lambda}{M(\left\Vert u_n\right\Vert_{s,q} ^{q})}\int_\Omega \dfrac{u_n \varphi}{\left| x\right|^{2s} }dx+\dfrac{1}{M(\left\Vert u_n\right\Vert_{s,q} ^{q})} \int_\Omega f_n \varphi dx.
$$
Then,$$ (\beta-\dfrac{\lambda}{M(\left\Vert u_n\right\Vert_{s,q} ^{q})})\int_\Omega \dfrac{\varphi u_n}{\left| x\right|^{2s}}dx+\int_\Omega u_n dx\leq \dfrac{1}{m_0}\int_\Omega f_n \varphi dx\leq \dfrac{1}{m_0}\int_\Omega f_n \left| x\right|^{-\g_\b} dx. $$
Now, since $M(||u_n||_{s,q}^{q})\longrightarrow +\infty$, we can choose $n_0$ large enough such that
$\beta-\dfrac{\lambda}{M(||u_n||_{s,q} ^{q})}>\varepsilon_0>0$ if $n\ge n_0$.
Hence, for $n\ge n_0$,
$$\beta_0 \int_\Omega \dfrac{\varphi u_n}{\left| x\right|^{2s}}dx+\int_\Omega u_n dx\leq \dfrac{1}{m_0}\int_\Omega f \left| x\right|^{-\theta_\varepsilon} dx.$$
Recall that $\displaystyle{\int_\Omega f \left| x\right| ^{-\theta} dx}<c$, then we fix now $\beta$ small enough such that $\theta_\varepsilon\le \theta$. Thus $\dyle\int_\Omega f \left| x\right|^{-\theta_\beta}<\infty$.
Thus
\begin{equation}\label{end1}
\beta_0 \int_\Omega \dfrac{\varphi u_n}{\left| x\right|^{2s}}dx+\int_\Omega u_n dx<m^{-1}_0c.
\end{equation}
From which we deduce that $\bigg\|\dfrac{u_n}{|x|^{2s}}\bigg\|_{L^1(\Omega)}<m^{-1}_0c$ and the claim follows.

Now going back to \eqref{tlemcen} and using the fact that the sequence $\{g_n\}_n$ is bounded in $L^1(\O)$, we reach that the sequence $\{u_n\}_n$ is bounded in
$W^{s,\s}_0(\Omega)$ for all $\s<\frac{N}{N-s}$, in particular for $\s=q$. Thus $\left\|u_n \right\|_{s,q }<C_1$ and then $M(\left\Vert u_n\right\Vert_{s,q} ^{q})<\bar{C}$, a contradiction with the main hypothesis. Thus $\{M(||u_n||_{s,q}^{q})\}_n$ is bounded in $\re^+$.

Hence, we get the existence of $\bar{u}\in W^{s,q}_0(\Omega)$ such that, up to a subsequence, $u_n\rightharpoonup \bar{u}$ weakly in $W^{s,q}_0(\Omega)$,
$u_n\to\bar{u}$ strongly in $L^\s(\O)$ for all $\s<q^*_s$ and $u_n\to \bar{u}\:\:a.e \mbox{  in }\Omega$.

Setting $\alpha_n:=M(\left\|u_n \right\|^q_{s,q})$, then we can assume that $\a_n\longrightarrow \bar{\alpha}\geq m_0.$

Let us show that $\dfrac{u_n}{|x|^{2s}}\to \dfrac{u}{|x|^{2s}}$ strongly in $L^1(\O)$. It is clear that $\dfrac{u_n}{|x|^{2s} }\to \dfrac{\bar{u}}{|x|^{2s}}$  a.e. in $\Omega.$ Using the fact that $u_n\to\bar{u}$ strongly in $L^\s(\O)$ for all $\s<q^*_s$, we reach that
$$
\dfrac{u_n}{|x|^{2s} }\to \dfrac{\bar{u}}{|x|^{2s}}\mbox{  strongly  in  }L^1(\O\backslash B_R(0) \mbox{  for all   }R>0.
$$
Hence to conclude we have just to show that
$$
\dfrac{u_n}{|x|^{2s} }\to \dfrac{\bar{u}}{|x|^{2s}}\mbox{  strongly  in  }L^1(B_R(0)) \mbox{  for some }R>0.
$$
Recall that from \eqref{end1}, we have $\dyle\int_\Omega \dfrac{\varphi u_n}{\left| x\right|^{2s}}dx<C$ for all $n$. By the fact that $\varphi\ge C_\beta|x|^{-\theta_\e}$ in $B_R(0)$, it holds that
$$
C_\beta R^{-\theta_\beta}\int_{B_R(0)}\dfrac{u_n}{\left| x\right|^{2s}}\le C\mbox{  for all   }n.
$$
Fix $R>0$ to be chosen later and consider $E\subset B_R(0)$ a measurable set. Then
\begin{eqnarray*}
\int_{E}\dfrac{u_n}{\left| x\right|^{2s}}dx &= & \int_{E\cap B_R(0)}\dfrac{\varphi u_n}{\left| x\right|^{2s}}dx+\int_{E\backslash B_R(0)}\dfrac{\varphi u_n}{\left| x\right|^{2s}}dx\\
&\le & C_\beta R^{\theta_\beta}+\int_{E\backslash B_R(0)}\dfrac{\varphi u_n}{\left| x\right|^{2s}}dx.
\end{eqnarray*}
Let $\e>0$, then we can choose $\d>0$ such that if $|E|<\d$ and $n\ge n_0$, then
$$
\int_{E\backslash B_R(0)}\dfrac{\varphi u_n}{\left| x\right|^{2s}}dx\le \frac{\e}{2}.
$$
It is clear that we can choose $R$ small enough such that
$$
C_\beta R^{\theta_\beta}\le \frac{\e}{2}.
$$
Combining the above estimates, it holds that for $n\ge n_0$ and if $|E|\le \d$, then
$$
\int_{E}\dfrac{u_n}{\left| x\right|^{2s}}dx\le \e.
$$
Hence, using Vitali's lemma, we reach that
$\dfrac{u_n}{|x|^{2s}}\to \dfrac{u}{|x|^{2s}}$ strongly in $L^1(\O)$.

Reminding the definition of $g_n$ in \eqref{tr} and using the compactness result in Theorem \ref{compa} we conclude, up to a subsequence, $u_n\to \bar{u}$ strongly in $W^{\s,q}_0(\Omega)$ for all $\s<\frac{N}{N-s}$. Thus
$$
M(\left\|u_n \right\|^q_{s,q})\to M(\left\|\bar{u}\right\|^q_{s,q})\mbox{  as   }n\to \infty.
$$
Therefore we conclude that $\bar{u}$ satisfies
$$(-\Delta) ^{s}\bar{u}=\dfrac{\lambda}{M(\left\Vert \bar{u} \right\Vert_{s,q} ^{q})}.\dfrac{\bar{u}}{\left| x\right|^{2s} }+\dfrac{f}{M(\left\Vert \bar{u} \right\Vert_{s,q} ^{q})},$$ and then $\bar{u}$ is solution of Problem $({P}^{\lambda}_{s,q})$.
\end{pf}
\begin{remark}
If $f\in L^\s(\O)$ for $\s>1$, then using H\"older inequality we can prove that the condition stated in Theorem \ref{generalq} holds for $f$ for some $\theta>0$. Thus problem $({P}^{\lambda}_{s,q})$ has a solution for all $f\in L^\s(\O)$ with $\s>1$ and for all $\l>0$.
\end{remark}

\section{Existence results in the case $f(x,u):=u^p+g(x)$.}\label{gen}
Let consider now the nonlinear problem
\begin{equation}\label{non1}
\left\{
\begin{array}{rcll}
M(\left\Vert u\right\Vert_{s,2}^{2})(-\Delta) ^{s}u & = & \l\dfrac{u}{\left| x\right|^{2s} }+u^p + \mu g&\text{ in } \Omega, \\
u & >& 0 & \text{ in } \Omega,\\
u & = & 0 &\text{ in } \mathbb{R}^{N}\setminus \Omega.
\end{array}
\right.
\end{equation}
In the case where $M\equiv 1$, then, as in the local case, for $\l<\L_{N,s}$ fixed, it is possible to show the existence of a critical value $p_*>2^*_s-1$ such that the problem \eqref{non1} has a nonnegative super-solution if and only if $p<p_*$. This can be obtained using suitable radial computation, according the behavior of any super-solution near the origin and a suitable test function, see  \cite{BMP} and \cite{BDT}.

The goal of this section is to check that, under suitable hypothesis on $g$ and $p$, the existence of a solution to problem \eqref{non1} for all $\l>0$.
The first existence result is the following.
\begin{Theorem}\label{non11}
Let $g\in L^\s(\O)$ with $\s\ge \frac{2N}{N+2s}$ and suppose that $p<2^*_s-1$, then there exists $\mu^*>0$ such that for all $\l>0$ and for all $\mu<\mu^*$, problem \eqref{non1} has a solution $u\in H^s_0(\O)$.
\end{Theorem}
\begin{proof}
In the case $M=1$, existence of solution is proved under the condition $\l\le \L_{N,s}$ and the comparison principle using a suitable radial supersolution. However this argument is not applicable in our case and we have to use a different argument in order to show the existence of a solution. The main idea is to use a suitable fixed point argument.

Fix $p<r<2^*_s$, we can show the existence of $\mu^*>0$ such that if $\mu<\mu^*$, then the algebraic equation
$$
C(l+\lambda^{\ast} \|f\|_{L^{\frac{2N}{N+2s}}(\O)})=l^{\frac{1}{p}},
$$
where $C$ is a universal constant that will be chosen later.

Define the set
$$
E:= \left\{ \varphi \in W_0^{s,1}(\Omega) : ||\varphi||_{L^r(\O)}\le l^{\frac 12}\right\},
$$
Then $E$ is a closed convex set of $W_0^{s,1}(\Omega)$.

Now, we define $T: E \to W_0^{s,1}(\Omega)$ by $T(\varphi) = u$, where $u$ is the weak solution to
\begin{equation} \label{eqf}
\left\{
\begin{array}{rcll}
M(||u||_{s,2}^{2})(-\Delta) ^{s}u & = & \l\dfrac{u}{\left| x\right|^{2s}} + \varphi_+^p+ \mu g(x)\,, & \mbox{ in } \Omega,\\
u & = 0\,, &  \mbox{ in } \ren\setminus \Omega.
\end{array}
\right.
\end{equation}
It is clear that if $u = T(u)$, then $u$ solves problem \eqref{non1}.  Hence, to prove Theorem \ref{non11}, we shall show that $T$ has fixed point belonging to $W_0^{s,2}(\Omega) \cap \mathcal{C}^{0,\alpha}(\Omega)$ for some $\alpha > 0$.

The proof will be given in several steps.

{\bf Step 1: $T$ is well defined.}

Since $\varphi\in E$, then $\varphi_+^p+g\in L^{\frac{2N}{N+2s}}(\O)$. Thus using Theorem \ref{thf1} and the Remark \ref{rem},  we get the existence and the uniqueness of $u\in H^s_0(\O)$ such that $u$ solves the problem \eqref{eqf}. Hence $u\in W^{s,1}_0(\O)$ and then $T$ is well defined.

{\bf Step 2: $T(E) \subset E$.}

Let $\varphi  \in E$, we define $u = T(\varphi)$. By the regularity result in Proposition \ref{rrem}, it follows that
\begin{equation} \label{ineq10}
\begin{aligned}
||u||_{L^r(\O)}
& \leq C \big\| \varphi_+^p+ \lambda g(x) \big \|_{L^{\frac{2N}{N+2s}}(\O)}\\
& \leq C(||\varphi||^{\frac{p}{r}}_{L^{r}(\Omega)}+ \lambda^{\ast} \|f\|_{L^{\frac{2N}{N+2s}}(\O)})\\
& \leq C(l+\lambda^{\ast} \|f\|_{L^{\frac{2N}{N+2s}}(\O)})\le l^{\frac{1}{p}}.
\end{aligned}
\end{equation}
Since, we have proved previously that $u \in W_0^{s,1}(\Omega)$, then using the definition of $l$, we conclude that $u \in E$ and so, that $T(E) \subset E$.

{\bf Step 3: $T$ is continuous.}

Let $\{\varphi_n\}_n \subset E$ be such that $\varphi_n \to \varphi$ in $W_0^{s,1}(\Omega)$. Consider $u_n = T(\varphi_n)$ and $u = T(\varphi)$.
Define
$$
 h_n(x) := (\varphi_n)_+^p + \lambda h(x),
$$
then to get the desired result, we have just to show that $h_n\to h$ strongly in $L^{\frac{2N}{N+2s}}(\O)$ where
 $$
 h(x) := (\varphi)_+^p + \lambda g(x).
 $$
 Using Sobolev inequality in $W^{s,1}_0(\O)$ and since $\varphi_n \to \varphi$ strongly in $W_0^{s,1}(\Omega)$, then
 $\varphi_n \to \varphi$ strongly in $L^{\frac{N}{N-1}}(\O)$. Since $\{\varphi_n\}_n$ is bounded in $L^r(\O)$, then using an interpolation arguments it holds that $\varphi_n \to \varphi$ strongly in $L^{\s}(\O)$ for all $\s<r$ and in particular in $L^{\frac{2N}{N+2s}}(\O)$. Hence $u_n\to u$ strongly in $H^{s}_0(\O)$ and then the continuity of $T$ follows.

 {\bf Step 4: $T$ is compact.}

Let $\{\varphi_n\} \subset E$ be a bounded sequence in $W_0^{s,1}(\Omega)$ and define $u_n = T(\varphi_n)$. We will prove that, up to a subsequence, $u_n \to u$ in $W_0^{s,1}(\Omega)$ for some $u \in W_0^{s,1}(\Omega)$.
\medbreak

Recall that $\{\varphi_n\}_n \subset E$, then as in the proof of the {\bf Step 3}, we reach that $\{\varphi_n\}_n$ is a bounded sequence in $L^{\frac{2N}{N+2s}}(\O)$. Now, setting
$$h_n(x):= (\varphi_n)_+^p + \lambda g(x),$$
we have that $\{h_n\}$ is a bounded sequence in $L^{\frac{2N}{N+2s}}(\O)$. The result then follows from the compactness result in Proposition \ref{compa}.

As a conclusion, and since $E$ is a closed convex set of $W_0^{s,1}(\Omega)$ and, by the previous steps, we can apply the Schauder fixed point Theorem to get the existence of $u \in E$ such that $T(u) = u$. Thus, we conclude that problem \eqref{non1} has a weak solution for all $0 < \lambda \leq \lambda^{\ast}$. It is clear that $u \in H^s_0(\Omega)$.

\end{proof}

\end{document}